\documentclass[reqno]{amsart}
\usepackage{amsfonts,amssymb}
\usepackage{enumerate}

\newcommand{\Z}{\mathbb{Z}}
\newcommand{\N}{\mathbb{N}}
\newcommand{\E}{\mathcal{E}}
\newcommand{\ES}{\mathcal{E}^*}
\newcommand{\es}{E^*}

\newtheorem{thm}{Theorem}[section]
\newtheorem{lem}[thm]{Lemma}

\theoremstyle{definition}

\theoremstyle{remark}

\author{Jing-Jing Huang \and Robert C. Vaughan}
\address{
JJH: Department of Mathematics, McAllister Building,
Pennsylvania State University, University Park, PA 16802-6401,
U.S.A.}
\email{huang@math.psu.edu}
\thanks{}

\address{
RCV: Department of Mathematics, McAllister Building,
Pennsylvania State University, University Park, PA 16802-6401,
U.S.A.}
\email{rvaughan@math.psu.edu}
\thanks{}

\begin{document}

\title
{On the Exceptional Set for Binary Egyptian Fractions}

\begin{abstract}
For fixed integer $a\ge3$, we study the binary Diophantine equation $\frac{a}n=\frac1x+\frac1y$ and in particular the number $E_a(N)$ of $n\le N$ for which the equation has no positive integer solutions in $x, y$. The asymptotic formula
$$E_a(N)\sim C(a) \frac{N(\log\log N)^{2^{m-1}-1}}{(\log N)^{1-1/2^m}}$$ 
as $N$ goes to infinity, is established in this article, and this improves the best result in the literature dramatically. The proof depends on a very delicate analysis of the underlying group structure.  
\end{abstract}
\maketitle

\section{Introduction} \label{s1}
\noindent Let $a$ be a fixed positive integer.  We consider the binary Diophantine equation
\begin{equation}
\label{e1.1}
\frac{a}{n}=\frac1x+\frac1y
\end{equation}
and denote by $R(n;a)$ the number of pairs of positive integer solutions $(x,y)$ satisfying the equation \eqref{e1.1}.
A good deal is now known about the average behaviour of $R(n;a)$.  See \cite{CDFHP}, \cite{HV1} and \cite{HV2} for details.  In this paper, we are concerned with the number of $n$ such that the equation \eqref{e1.1} is not soluble in positive integers $x$ and $y$ and to this end we define
$$\E_a=\{n\in \N:R(n;a)=0\}.$$
Clearly both $\E_1$ and $\E_2$ are empty.  When $a\geq 3$ the structure of $\E_a$ is more delicate and of great interest. In this paper, we investigate the asymptotic size of $\E_a$.  Thus we define
$$\E_a(N)=\{n\in \E_a:n\leq N\}$$
and
$$E_a(N)=\#\E_a(N).$$

In 1985, G. Hofmeister and P. Stoll \cite{HS} proved that the set $\E_a$ has asymptotic density 0, and more precisely that
$$E_a(N)\ll_a \frac{N}{(\log N)^{1/\phi(a)}}.$$
For $a=5$ and $a\ge 7$ this bound is far from the truth.  Their method is based on the observation that if the equation \eqref{e1.1} is insoluble, then $n$ is not divisible by any prime of the form $p\equiv -1\pmod a$.  Thus a simple application of Selberg's upper bound sieve gives the stated bound.  However when $a=5$ or $a\ge 7$ the bulk of the $n$ deficient in such prime factors nevertheless have a representation.

The following theorem establishes the precise asymptotic behaviour of $E_a(N)$.
 
\begin{thm}\label{t1.1}
For fixed $a\ge 3$, let $2^{\gamma_0}p_1^{\gamma_1}p_2^{\gamma_2}\cdots p_k^{\gamma_k}$ be the canonical decomposition of $a$ and define $m$ and $\delta$ by
$$2^m||\gcd(\delta,p_1-1, p_2-1, \cdots, p_k-1)$$
and
\[
\delta=
\left\{
\begin{array}{lll}
0, & \text{if} & \gamma_0\le 1,\\
2, & \text{if} & \gamma_0\ge 2.
\end{array}
\right.
\]
Then we have
\[
E_a(N)\sim C(a) \frac{N(\log\log N)^{2^{m-1}-1}}{(\log N)^{1-1/2^m}}
\]
where 
$C(a)$ is a positive constant depending only on $a$.
\end{thm}

In order to establish this theorem we need first to investigate the underlying structure of $\E_a$, and we embark on this in Section \ref{s2}.  The case when $a$ is a prime power is somewhat easier to understand and, having established some preliminary lemmata in \ref{s2.1}, we consider this case in Section \ref{s2.2}.  This then leads into a discussion of the general case in \ref{s2.3}. 

In Section \ref{s3} the main analytic input is introduced, and it is convenient to base this on an arithmetical application of a theorem of Delange.  Delange's theorem is a refinement of the Wiener--Ikehara theorem and is qualitative in nature.  In particular it does not give an explicit error term.  By using instead a method allied to that leading to the strongest known unconditional error term in the prime number theorem it would be possible to give a quantitative error term in Lemma \ref{lem3.2} of a similar quality.  However whilst this would be quite routine in nature there would be many detailed complications and more importantly the extra effort would not lead to any further illumination of the central problem of this paper in that a greater loss in the error term appears at a later stage of our argument.  We are happy to leave this approach as an exercise to the reader.  

The proof of the main theorem is completed in Section \ref{s4} through a suitable combination of Sections \ref{s2} and \ref{s3}.    

Throughout this paper, we reserve the letters $p$, $q$ and $r$ for prime numbers and calligraphic letters for sets and sequences. In particular, if $\mathcal{A}\subseteq\N$ we denote by $\mathcal{A}(N)$ the subset of $\mathcal{A}$ with elements less than or equal to $N$ and $|\mathcal{A}(N)|$ denotes the cardinality of $\mathcal{A}(N)$. We also use Vinogradov's ``$\ll$" notation, namely when we write $f(x)\ll g(x)$ we mean $|f(x)|\le C g(x)$ for some absolute constant $C$ and sufficiently large $x$. And accordingly, the notation $\ll_a$ means that the implicit constant defined above depends on some parameter $a$. 

\section{The Structure of $\E_a$} \label{s2}

\subsection{Some elementary lemmata}\label{s2.1}

It is more convenient to work with the notations $\ES_a$, $\ES_a(N)$ and $\es_a(N)$, defined as follows:
$$\ES_a=\{n\in\E_a:(n,a)=1\}$$
and $\ES_a(N)$ and $\es_a(N)$ can be defined accordingly.  Then we have immediately the following.
\begin{lem} \label{lem2.1}
We have
$$\E_a(N)=\bigcup_{d|a}\ES_{a/d}(N/d)$$
and hence have
$$E_a(N)=\sum_{d|a}\es_{a/d}(N/d).$$
\end{lem}
Thus the structure of $\E_a$ can be deduced readily from that of $\ES_a$.  Henceforward we assume that $n$ is a positive integer coprime to $a$, unless otherwise stated.  

The starting point of our argument is the following elementary lemma.

\begin{lem}\label{lem2.2}
The equation \eqref{e1.1} with $(a,n)=1$ is soluble in positive integers if and only if there exists a pair of coprime factors $u$ and $v$ of $n$  such that $a|u+v$.
\end{lem}

\begin{proof}
If the equation \eqref{e1.1} is soluble, then we rewrite it as $axy=n(x+y)$, let $(x,y)=l$ and write $x=ul$ and $y=vl$ with (u,v)=1. Thus $aluv=n(u+v)$. Then, as $(a,n)=1$ and $(uv,u+v)=1$, we have $uv|n$ and $a|u+v$.

In the opposite direction, we write $u+v=a a'$ and $n=u v n'$, so that  $\frac{a}{n}=\frac1{a' n' u}+\frac1{a' n' v}$.  
\end{proof}

This lemma suggests that the solubility of equation \eqref{e1.1} solely depends on the residue classes of factors of $n$ modulo $a$, and hence depends on the residue classes of prime factors of $n$ modulo $a$, which naturally leads our discussion to the distribution of prime factors of $n$ in the multiplicative group $(\Z/a\Z)^*$.  

\subsection{The case that $a$ is a power of odd prime}\label{s2.2}

We consider the case that $a=p^\gamma$ is a power of odd prime in this subsection and come back to the general case later. This strategy fits with both the motivational purpose and the presentational purpose. Let $G$ denote the cyclic group $(\Z/a\Z)^*$ of reduced residue classes modulo $a=p^\gamma$, and let $H$ be the maximal subgroup of $G=(\Z/a\Z)^*$ with cardinality $|H|$ being odd, namely $H$ is the maximal subgroup of $G$ such that $\overline{-1}\notin H$ and clearly such a group is unique. Here and throughout this article $\overline{i}$ means the residue class $i\pmod{a}$, if there is no ambiguity about the modulus $a$ in the context. Now let $\phi(a)=2^m d$ with $d$ being an odd number. If we fix a primitive root $g$ modulo $a$, then 
\begin{equation}\label{e2.1}
G=\{g, g^2, g^3, \ldots, g^{2^m d}\}
\end{equation}
and
\begin{equation}\label{e2.2}
H=\{g^{2^m}, g^{2\cdot2^m}, g^{3\cdot2^m}, \ldots, g^{d\cdot 2^m}\},
\end{equation}
by which one readily verifies that $\overline{-1}\notin H$ since $g^{\phi(a)/2}\equiv-1\pmod{a}$. Hence we have the index $[G:H]=2^m$ and $|H|=\frac{\phi(a)}{2^m}=d$.

Essentially the structure of $\ES_a$ is that any $n\in \ES_a$ can have arbitrarily many prime factors lying in the residue classes in $H$ but can have at most a bounded number of prime factors lying outside $H$. It is this observation that renders the counting function of $\ES_a$ susceptible to an analytic argument.

\begin{lem}\label{lem2.3} 
We have the following inclusion relation of sets 
$$\{n\in\N:q|n \text{ with } q \text{ being prime } \Rightarrow \overline{q}\in H\}\subseteq \ES_a.$$
\end{lem}  
\begin{proof}
For any $n$ on the left hand side, and for any pair of coprime positive integers $u$ and $v$ with $uv|n$ we have $\overline{u}, \overline{v}\in H$ in light of the fact that $H$ is a group. Since $\overline{-1}\notin H$, we have $\overline{-v}\notin H$ and hence $\overline{u}\neq\overline{-v}$, in other words $a\nmid u+v$. Now Lemma \ref{lem2.3} follows from Lemma \ref{lem2.2}. 
\end{proof}

The next lemma is central to our understanding of the structure of $\ES_a$.

\begin{lem}\label{lem2.4}
Let $m\ge1$ and $\mathcal{G}$ denote the additive group $\Z/(2^m\Z)$, let $\{e_j\}_1^t$ be a sequence with $t$ nonzero elements of $G$ (i.e., repeated elements are allowed in $\{e_j\}$), and form the set
$$\mathcal{S}=\left\{\sum_{j=1}^t\delta_j e_j:\delta_j\in\{-1,0,1\}\right\}.$$
\begin{enumerate}[(i)]
\item If $t\ge2^{m-1}$, then $\overline{2^{m-1}}\in\mathcal{S}$. Namely, as long as the length of $\{e_j\}$ exceeds $2^{m-1}$, for whatever choices of the elements $e_j$, one can always find a partial sum, as in the definition of $\mathcal{S}$, such that it is equal to $\overline{2^{m-1}}$.
\item If $t=2^{m-1}-1$, then the corresponding set $\mathcal{S}$ does not contain $\overline{2^{m-1}}$ if and only if the sequence $\{e_j\}_1^t$ satisfies  $e_j\equiv \pm e\pmod{2^m }$ for each $j$ and some fixed $e\in(\Z/(2^m)\Z)^*$. 
\end{enumerate}
\end{lem}

\begin{proof}  We note that $\mathcal S$ is a subset of $\mathcal{G}$.  The proof is by induction on $m$.  The initial case $m=1$ is trivial.  Thus we can suppose that $m\ge1$ and that the conclusion is true for $m$. Consider $\mathcal{G}=\Z/(2^{m+1}\Z)$ and a sequence $\{e_j\}_1^{2^{m}}\subseteq\mathcal{G}$. By the induction assumption, we know that there exist $\delta_j\in\{-1,0,1\}$ for $1\le j\le 2^m$ such that
$$s_1:=\sum_{j=1}^{2^{m-1}}\delta_j e_j\equiv 2^{m-1}\pmod{2^m}$$
and 
$$s_2:=\sum_{j=2^{m-1}+1}^{2^m}\delta_j e_j\equiv 2^{m-1}\pmod{2^m}.$$
Choose $u_i$ so that $s_i=2^{m-1}+u_i2^m$ for $i\in\{1,2\}$.  Then by considering separately the cases when the $u_i$ are of the same or differing parity it follows that either $s_1+s_2$ or $s_1-s_2$ is congruent to $2^{m}$ modulo $2^{m+1}$.  This establishes (i).  The proof of (ii) is similar but a little more elaborate.  If there is an $e\in(\Z/(2^m)\Z)^*$ such that $e_j\equiv\pm e\pmod{2^m}$ for every $j$, then regardless of the choice of $\delta_j$ we have $\sum_{j=1}^t\delta_je_j\equiv \pm ue\pmod{2^m}$ where $|u|\le 2^{m-1}-1$.  Thus $\overline{2^{m-1}}\not\in\mathcal S$.  Thus it remains to consider the situation when $\overline{2^{m-1}}$ is not contained in $\mathcal S$.  As before, we argue by induction on $m$.  When $m=1$ we have $t=0$ and $\mathcal S$ is empty so the conclusion is trivial.  When $m=2$ we have $t=1$ and $2^{m-1}=2$, and so $e_1\not\equiv 0$ or $2\pmod 4$ and we are done.  Now suppose that the conclusion holds for a given value of $m\ge2$ and consider the case with $m$ replaced by $m+1$.  That is, we suppose that $\overline{2^m}$ is not contained in $\mathcal S$ and will deduce that there is an $e\in(\Z/(2^{m+1})\Z)^*$ such that each $e_j$ satisfies $e_j\equiv\pm e\pmod{2^{m+1}}$.  We now form the partial sums 
$$s_1:=\sum_{j=1}^{2^{m-1}-1}\delta_j e_j$$
and 
$$s_2:=\sum_{j=2^{m-1}}^{2^m-1}\delta_j e_j.$$
By (i) and the inductive hypothesis if there is no $e$ such that $e_j\equiv\pm e\pmod{2^m}$ for $1\le j\le 2^{m-1}-1$, where $e\in(\Z/(2^m)\Z)^*$, then there is a choice of the $\delta_j$ such that
$$s_2\equiv 2^{m-1}\pmod{2^m}$$
and
$$s_1\equiv 2^{m-1}\pmod{2^m}.$$
Thus if there is no such $e$, then as before one of $s_1\pm s_2\equiv 2^m\pmod{2^{m+1}}$, which we have expressly excluded.  Thus there is such an $e$.  Moreover we can repeat the argument with every permutation of the $e_j$.  Thus we can conclude that there is an $e$ such that $e_j\equiv\pm e\pmod{2^m}$ for $1\le j\le 2^m-1$, where $e\in(\Z/(2^m)\Z)^*$.  In other words
$$e_j\equiv \pm e\text{ or }\pm (e+2^m)\pmod{2^{m+1}}.$$
Now we may conclude that either all the $e_j$ are congruent to $\pm e$ or they are congruent to $\pm (e+2^m)$, because if, say, $e_1\equiv \pm e\pmod{2^{m+1}}$ and $e_2\equiv \pm(e+2^m)\pmod{2^{m+1}}$ then either $e_1+e_2$ or $e_1-e_2$ is $2^m\pmod{2^{m+1}}$, contradicting  $\overline{2^m}\not\in\mathcal S$.
\end{proof}
A weaker version of the lemma in which one replaces the exact lower bound   $2^{m-1}$ of $t$ in part (i) by the crude bound $(2^m-1)(2^{m-1}-1)+1$, would follow by a direct application of the pigeonhole principle. 
An extension of this lemma to general modulus (not necessarily a power of 2) could be formulated and then proved by Kneser's theorem (see chapter 1 in \cite{HR}), which is, however, not of direct relevance to the purpose of this memoir. Nevertheless, it would be of essence if one desires to establish the second order term for the asymptotics in Theorem \ref{t1.1}.

Having established the necessary preliminaries, we are poised to reveal the structure of $\ES_a$ when $a=p^\gamma$ is a power of an odd prime.  

\begin{lem} \label{lem2.6}
Let $\mathcal{P}$ be the sequence of prime factors of $n$, counted with multiplicity. And let $\mathcal{T}$ be the subsequence of prime $r$ in $\mathcal{P}$ with $\overline{r}\notin H$, then denote by $t$ the length of $\mathcal{T}$. Considering the projection map: $\Z\rightarrow\Z/a\Z$,  suppose the image of the sequence $\mathcal{P}$ contains $H$. 
\begin{enumerate}[(i)]
\item
If $t\ge2^{m-1}$, then $n\notin\ES_a$.
\item 
If $t=2^{m-1}-1$, then $n\in\ES_a$ if and only if every prime factor in $\mathcal{T}$ is congruent to $g^{e'}$  modulo $p^\gamma$ for a fixed primitive root $g \pmod{p^\gamma}$, and for some $e'$ such that $e'\equiv \pm e\pmod{2^m}$ with $e$ being a fixed odd number.    
\end{enumerate}
\end{lem}

\begin{proof}

Recall that $G$ and $H$ are give by \eqref{e2.1} and \eqref{e2.2} respectively, for a fixed primitive root $g$ modulo $a$. Denote $\mathcal{T}=\{r_j\}_1^t$. Let the sequence $\{e_j\}$ be such that $g^{e_j}\equiv r_j\pmod{a}$. By the assumption $\overline{r_j}\notin H$ we know $e_j\not\equiv 0 \pmod{2^m}$, for $1\le j\le t$. Let $\mathcal{G}=\Z/2^m\Z$. Now $\{e_j\}$ can be viewed as a sequence of nonzero elements in $\mathcal{G}$. 
Clearly we see Lemma \ref{lem2.4} gets into play here. More precisely, 
when $t\ge2^{m-1}$, there exist $\delta_j\in\{-1, 0, 1\}$ such that
$$\sum_{j=1}^t\delta_j e_j\equiv 2^{m-1}\pmod{2^m}.$$
This is equivalent to 
$$\sum_{j=1}^t\delta_j e_j\equiv b 2^{m-1}\pmod{2^m d},$$
for some odd number $b$ such that $1\le b\le d$. Hence
$$\sum_{j=1}^t\delta_j e_j+(d-b)2^{m-1}\equiv  2^{m-1}d\pmod{2^m d}.$$
Translating this using multiplicative language, we know that 
$$g^{\frac{d-b}2\cdot2^m}\prod_{j=1}^t(g^{e_j})^{\delta_j}\equiv  g^{2^{m-1}d}\equiv-1\pmod{a}.$$
By assumption there exists $q\in \mathcal{P}$ such that $q\equiv g^{\frac{d-b}2\cdot2^m}\pmod{a}$. On the other hand, $g^{e_j}\equiv r_j\pmod{a}$ and $q\prod_{j=1}^t r_j|n$. Hence
there exist two coprime divisors $u$ and $v$ of $n$, such that $\frac{u}v\equiv -1\pmod{a}$ namely $u+v\equiv0\pmod{a}$. By Lemma \ref{lem2.2}, we know $n\notin\ES_a$. This proves part (i).

For part (ii), the necessity of the condition follows by exactly the same argument as above,  keeping in mind that Lemma \ref{lem2.4} still plays an important role. Now in order to prove the sufficiency, we just need to reverse the above argument and argue by contradiction. 
(Notice that the condition $\overline{\mathcal{P}}$ contains $H$ is not needed in this direction.)
\end{proof}

Our next task naturally is to extend Lemma \ref{lem2.6} to general modulus. We will see how one can carry the arguments here to the general case only with some mild difficulties in the next subsection.
 
\subsection{The case for general $a$}\label{s2.3}
Now we treat the general case $a=2^{\gamma_0}p_1^{\gamma_1}p_2^{\gamma_2}\cdots p_k^{\gamma_k}$. Of course by Chinese remainder theorem we have the group isomorphism
$$(\Z/a\Z)^*\simeq(\Z/2^{\gamma_0}\Z)^*\times(\Z/p_1^{\gamma_1}\Z)^*\times\cdots
\times(\Z/p_k^{\gamma_k}\Z)^*.$$
As before, we still denote this group by $G$.
Here all the groups $(\Z/p^\gamma\Z)^*$ are cyclic when $p$ is an odd prime, but in general $(\Z/2^{\gamma_0}\Z)^*$ is not except that $\gamma_0\le 2$. For instance, $(\Z/2\Z)^*$ is trivial and $(\Z/4\Z)^*$ is isomorphic to $(\Z/2\Z,+)$. In particular, there is no difference between the cases $\gamma_0=0$ and $\gamma_0=1$ because they exert no influence to $G$. While, when $\gamma_0\ge3$, $(\Z/2^{\gamma_0}\Z)^*$ is a product of 2 cyclic groups with generators $-1\pmod{2^{\gamma_0}}$ and $5\pmod{2^{\gamma_0}}$ respectively, namely
$$(\Z/2^{\gamma_0}\Z)^*\simeq\langle\overline{-1}\rangle\times\langle
\overline{5}\rangle.$$
Apparently $|\langle\overline{-1}\rangle|=2$ and $|\langle\overline{5}\rangle|=2^{\gamma_0-2}$. 

Here, we still want to find a maximal subgroup $H$ of $G$ such that $-1\pmod{a}\notin H$. However, the issue here is that such subgroups of $G$ might not be unique. They can be easily constructed as follows. Let $G_i=(\Z/p_i^{\gamma_i}\Z)^*$, for $0\le i\le k$ and $H_i$ to be the maximal subgroup of $G_i$ such that $-1\pmod{p_i^{\gamma_i}}\notin H_i$. As we remarked before, $H_1, H_2, \cdots, H_k$ are unique but $H_0$ is not in general. In fact, $H_0$ is trivial if $\gamma_0\le2$ and is one of the two subgroups of index 2 in the ambient group $G_0$ if $\gamma_0\ge 3$. Recall our discussion in the cyclic case, hence $[G_i:H_i]=2^{m_i}$ for some positive integer $m_i$ and for all $1\le i\le k$. Moreover 
\[
[G_0:H_0]=\left\{
\begin{array}{lll}
1, & \text{if} & \gamma_0\le 1\\
2, & \text{if} & \gamma_0\ge 2.
\end{array}
\right.
\]
Now choose $m$ such that
\[
m=\left\{
\begin{array}{lll}
\displaystyle\min_{1\le i\le k}m_i, & \text{if} & \gamma_0\le 1\\
&&\\
1, & \text{if} & \gamma_0\ge 2,
\end{array}
\right.
\]
namely
$$2^m||\gcd(\delta,p_1-1, p_2-1, \cdots, p_k-1),$$
where
\[
\delta=
\left\{
\begin{array}{lll}
0, & \text{if} & \gamma_0\le 1,\\
2, & \text{if} & \gamma_0\ge 2.
\end{array}
\right.
\]

By definition, we have $m\ge1$. The subgroup $H$ as described above is one of the following groups with index $[G:H]=2^m$:
$$H_0\times G_1\times\cdots \times G_k, G_0\times H_1\times\cdots \times G_k, \cdots$$
in which we just replace the $i$-th component of $G$ by $H_i$ for $0\le i\le k$. We write $\phi(a)=2^m d$ and hence $|H|=\frac{\phi(a)}{2^m}=d$. Notice that $d$ is not necessarily odd in general.  

It's routine to prove the following lemma (see the proof of Lemma \ref{lem2.3}).

\begin{lem}\label{lem2.7}
We have the following inclusion relation of sets
$$\{n\in\N:q|n \text{ with } q \text{ being prime } \Rightarrow \overline{q}\in H\}\subseteq \ES_a.$$
\end{lem}

The next lemma is crucial for our arguments.

\begin{lem}\label{lem2.8}
Let $H'$ be a subset of $G$ with cardinality $|H'|\ge d$. And suppose for each $h'\in H'$, there are at least $\phi(a)$ many (counted with multiplicity) prime factors $q$ of $n$ satisfying $q\equiv h'\pmod{a}$. Then $n\notin\ES_a$ unless $H'=H$ for some subgroup $H$ defined above.
\end{lem}
\begin{proof}
The proof still relies on Lemma \ref{lem2.2}. Actually, by Lemma \ref{lem2.2}, if we can find a divisor of $n$ which is congruent to $-1\pmod{a}$, then $n\notin \ES_a$. Now our argument goes roughly as follows, the fact that $n$ has sufficiently many primes factors lying in sufficiently many different reduced residue classes in $G$, forces $n$ to have at least one divisor lying in the residue class $-1\pmod{a}$ unless $H'$ is one of the above subgroups of $G$. To make this statement rigorous,  let $H''$ be the set of all the residue classes of divisors of $n$ in $G$. Then $\overline{1}\in H''$ and for any two elements $h_1,h_2\in H'$, we have $h_1^{-1}=h_1^{\phi(a)-1}\in H''$ and $h_1 h_2\in H''$. This means that $H''$ contains the subgroup $\langle H'\rangle$ generated by the elements in $H'$ and in particular this subgroup has cardinality at least $|H'|\ge d$. However our $H$ is maximized such that $\overline{-1}\notin H$, which implies either that $\overline{-1}\in\langle H'\rangle$ and hence $\overline{-1}\in H''$, or that $H'$ itself is a maximal subgroup such that $\overline{-1}\notin H'$. In the former case, we have $n\notin \ES_a$ by Lemma \ref{lem2.2} and in the latter case, we know by Lemma \ref{lem2.7} that $n\in \ES_a$.
\end{proof}

Now we need an analogue of Lemma \ref{lem2.6} for the general case. Here we need to pay special attention to the power of $2$ dividing $a$. When $\gamma_0\ge2$, we have $m=1$, which is sort of the ``worst" case, for $E_a(N)$ is largest possible. 
\begin{lem}\label{lem2.9}
Suppose $H=G_0\times G_1\times \cdots \times H_i \times \cdots \times G_k$ is a subgroup of $G$ defined as above. Let $\mathcal{P}$ be the sequence of prime factors of $n$ (counted with multiplicity). And let $\mathcal{T}$ be the subsequence of prime $r$ in $\mathcal{P}$ with $\overline{r}\notin H$. Then denote by $t$ the length of $\mathcal{T}$. Considering the projection map: $\Z\rightarrow\Z/a\Z$,  suppose the image of the sequence $\mathcal{P}$ contains $H$. 
\begin{enumerate}[(i)]
\item
If $t\ge2^{m-1}$, then $n\notin\ES_a$.
\item 
If $t=2^{m-1}-1$ and $m\ge2$ (in this case, $\gamma_0\le1$ and hence $G_0$ is trivial and in particular our $H_i$ here cannot be $H_0$), then $n\in\ES_a$ if and only if every prime factor in $\mathcal{T}$ is congruent to $g^{e'}$  modulo $p_i^{\gamma_i}$ for a fixed primitive root $g \pmod{p_i^{\gamma_i}}$, and for some $e'$ such that $e'\equiv \pm e\pmod{2^m}$ with $e$ being a fixed odd number. 
\end{enumerate}   
\end{lem}
\begin{proof}
Generally speaking the arguments in the proof of Lemma \ref{lem2.6} still work here. Nevertheless, one needs to make some changes accordingly.
It is trivial to verify the conclusions when $m=1$. So without loss of generality we assume $m\ge 2$ hence $1\le i\le k$.


Denote $\mathcal{T}=\{r_j\}_1^t$ and fix a primitive root $g$ modulo $p_i^{\gamma_i}$. Let the sequence $\{e_j\}$ be such that $g^{e_j}\equiv r_j\pmod{p_i^{\gamma_i}}$. By the assumption ${r_j}\pmod{a}\notin H$ namely $r_j\pmod{p_i^{\gamma_i}}\notin H_i$ we know $e_j\not\equiv 0 \pmod{2^m}$, for $1\le j\le t$. Let $\mathcal{G}=\Z/2^m\Z$. Now $\{e_j\}$ can be viewed as a sequence of nonzero elements in $\mathcal{G}$. 
Hence by Lemma \ref{lem2.4},
when $t\ge2^{m-1}$, there exist $\delta_j\in\{-1, 0, 1\}$ such that
$$\sum_{j=1}^t\delta_j e_j\equiv 2^{m-1}\pmod{2^m}.$$
After writing $\phi(p_i^{\gamma_i})=2^m d_i$ with $d_i$ odd.
This is equivalent to 
$$\sum_{j=1}^t\delta_j e_j\equiv b 2^{m-1}\pmod{2^m d_i},$$
for some odd number $b$ such that $1\le b\le d_i$. Hence
$$\sum_{j=1}^t\delta_j e_j+(d_i-b)2^{m-1}\equiv  2^{m-1}d_i\pmod{2^m d_i}.$$
Translating this using multiplicative language, we know that 
$$g^{\frac{d_i-b}2\cdot2^m}\prod_{j=1}^t(g^{e_j})^{\delta_j}\equiv  g^{2^{m-1}d_i}\equiv-1\pmod{p_i^{\gamma_i}}.$$
By assumption there exists $q\in \mathcal{P}$ such that 
\[
\left\{
\begin{array}{l}
\displaystyle q\equiv -\prod_{j=1}^t r_j^{-\delta_j}\pmod{p_j^{\gamma_j}},\quad 1\le j\le k, j\not=i\\
\\
q\equiv g^{\frac{d_i-b}2\cdot2^m} \pmod{p_i^{\gamma_i}}.
\end{array}
\right.
\]
Hence by Chinese remainder theorem, we know
$$q\prod_{j=1}^t r_j^{\delta_j}\equiv -1 \pmod{a}$$
namely there exist two coprime divisors $u$ and $v$ of $n$, such that $\frac{u}v\equiv -1\pmod{a}$ namely $u+v\equiv0\pmod{a}$. Again by Lemma \ref{lem2.2}, we know $n\notin\ES_a$. 

Part (ii) can be proved similarly (see the comment in the proof of Lemma \ref{lem2.6}).
\end{proof}

\section{The Analytic Inputs} \label{s3}
We need the following generalisation of Ikehara's Tauberian Theorem, which is due to Delange (\cite{De}, see also Theorem 7.15 in Tenenbaum \cite{Te}). This extends Ikehara's Theorem to the case of a singularity of mixed type, involving algebraic and logarithmic poles.  As usual we use $\sigma$ to denote the real part of the complex number $s$, and we define $l(s)=\log\frac{1}{s-1}$ for $\sigma>1$ by taking $l(2)=0$ and then defining $l(s)$ by continuous variation along the line segment joining $2$ to $s$.

\begin{lem}[Delange, 1954]\label{lem3.1}
Let $f(s)=\sum_{n=1}^{\infty}a_n n^{-s}$ be a Dirichlet series with non-negative coefficients, converging for $\sigma>1$. Suppose that $f(s)$ is holomorphic at all points of the line $\sigma=1$ other than $s=1$ and that, in the neighborhood of the this point and for $\sigma>1$, we have
$$f(s)=(s-1)^{-\omega-1}\sum_{j=0}^{t}g_j(s)\left(\log\left(\frac1{s-1}
\right)\right)^j+g(s),$$
where $\omega$ is some real number, and the $g_j(s)$ and $g(s)$ are functions holomorphic at $s=1$, the number $g_t(1)$ being non-zero. Then:
\begin{enumerate} [(i)]
\item if $\omega$ is not a negative integer, we have as $x\rightarrow\infty$
$$\sum_{n\le x}a_n\sim\frac{g_t(1)}{\Gamma(\omega+1)}x(\log x)^\omega(\log\log x)^t,$$
\item if $\omega=-m-1$ for a non-negative integer $m$ and if $t\ge1$, we have as $x\rightarrow\infty$
$$\sum_{n\le x}a_n\sim(-1)^m m! t g_t(1) x(\log x)^\omega(\log\log x)^{t-1}.$$
\end{enumerate}
\end{lem}

The following lemma is the key analytic ingredient of this paper. Essentially it plays the role of a sieve, but the upshot is that it produces asymptotics, not just an upper bound as almost all sieves do.
\begin{lem}\label{lem3.2}
Suppose $a$ is a positive integer, and let $\mathcal{B}=\{\overline{b_1},\ldots,\overline{b_w}\}$ be a subset of $(\Z/a\Z)^*$ with $w\ge 0$ elements, $\mathcal{C}=\{\overline{c_j}\}_1^t$ be a sequence of length $t$ with elements in $(\Z/a\Z)^*$ (elements could be repeated). And suppose further that $\mathcal{B}$ and $\mathcal{C}$ do not share common elements. Now let $\mathbb P$ denote the set of primes and define
$$\mathcal{A}=\mathcal{A}(\mathcal{B},\mathcal{C})=\{q_1q_2\ldots q_l r_1 r_2\ldots r_t:q_i\in\mathbb P, r_j\in\mathbb P,\overline{q_i}\in\mathcal{B},  \overline{r_j}=\overline{c_j}, l\ge0\}.$$
Then
\begin{enumerate}[(i)]
\item if $w\ge1$, we have as $x\to\infty$ 
$$|\mathcal{A}(x)|\sim C(a, \mathcal{B}, t) \frac{x(\log\log x)^t}{(\log x)^{1-w/\phi(a)}},$$
\item if $w=0$ and $t\ge 1$, we have as $x\to\infty$
$$|\mathcal{A}(x)|\sim C(a, t) \frac{x(\log\log x)^{t-1}}{\log x}.$$
\end{enumerate}
The constants $C(a,\mathcal{B},t)$ and $C(a,t)$ are positive and do not depend on the choices of the $\overline{c_j}$.
\end{lem}

\begin{proof}  Let
\[
a_n=\left\{
\begin{array}{cc}
1,& n\in \mathcal{A},\\
0,& n\notin \mathcal{A}.
\end{array}
\right.
\]
The set $\mathcal{A}(\mathcal{B},\mathcal{C})$ has a multiplicative structure, and this leads naturally to the following Dirichlet series

\begin{equation}\label{e3.1}
f(s)=\sum_{n=1}^{\infty}\frac{a_n}{n^s}=\prod_{\substack{q\in\mathbb{P} \\\overline{q}\in\mathcal{B}}}(1-1/q^s)^{-1}\prod_{j=1}^{t}\sum_{\substack{r\in\mathbb{P}\\\overline{r}=\overline{c_j}}}\frac1{r^s}
\end{equation}
which converges absolutely and locally uniformly in the region $\sigma>1$.

When $D(s)$ is a Dirichlet series which converges absolutely and locally uniformly for $\sigma>\sigma_0$, has an analytic continuation for $\sigma>\sigma_1$, is non-zero for $\sigma>\sigma_2$ and satisfies $\lim_{\sigma\rightarrow\infty} D(\sigma)=1$, we define $D(s)^{\alpha}$ for $\sigma>\max(\sigma_1,\sigma_2)$ and an arbitrary complex number $\alpha$ by $\exp(\alpha\log D(s))$ where we choose the principal value of $\log D(\sigma_3)$ for some suitably large $\sigma_3$ and then define $\log D(s)$ by continuous variation from $\sigma_3$ to $s$.

Let $e(\chi)=\frac1{\phi(a)}\sum_{\chi\bmod a}\chi(q)\bar\chi(b)$.  Then, by the orthogonality of Dirichlet characters, the product over $q$ on the right of the equation \eqref{e3.1} is
\begin{align*}
&\prod_{\overline{b}\in\mathcal{B}}\prod_{q}(1-1/q^s)^{-e(\chi)}\\
=&\prod_{\overline{b}\in\mathcal{B}}\prod_{\chi\bmod a}\left(L(s,\chi) g_1(s,\chi)\right)^{\frac{\bar\chi(b)}{\phi(a)}}
\end{align*}
where 
$$g_1(s,\chi)=\prod_q\frac{(1-\chi(q)/q^s)}{(1-1/q^s)^{\chi(q)}},$$
which converges absolutely when $\sigma>\frac12$, and hence has no zeros in that region. Thus $g_1(s,\chi)^{\frac{\bar\chi(b)}{\phi(a)}}$ is a well defined analytic function when $\sigma>\frac12$. 

Now the above product can be further rearranged as
\begin{equation}\label{e3.2}
L(s,\chi_0)^{\frac{\omega}{\phi(a)}}g_1(s)
\end{equation}
where $\omega$ is the cardinality of $\mathcal{B}$, $\chi_0$ is the principal character modulo $a$ and 
$$g_1(s)=\prod_{\overline{b}\in\mathcal{B}}\prod_{\substack{\chi\neq\chi_0 \\\mod a}}(L(s,\chi) g_1(s,\chi))^{\frac{\bar\chi(b)}{\phi(a)}}.$$
In particular 
$$g_1(1)\not=0.$$
Note that $g_1(1)$ may depend on the choice of $\mathcal{B}$.

It is well known that $L(s,\chi)$ has no zeros with $\sigma\ge1$ and has an analytic continuation to the whole complex plane.  Moreover, when $\chi$ is non-principal it is entire and when $\chi$ is a principal character $\chi_0$ it has a simple pole at $s=1$  and $(s-1)L(s,\chi_0)$ is entire.  Thus, when $\chi$ is non-principal ,
$$L(s,\chi)^{\frac{\bar\chi(b)}{\phi(a)}}$$
is analytic in the region $\sigma\ge1$ and hence so is $g_1(s)$.   

On the other hand, again by the orthogonality of Dirichlet characters the sum over $r$ on the right of the equation \eqref{e3.1} is
$$\frac1{\phi(a)}\sum_{\chi\bmod a}\bar\chi(-c_j)\sum_{p}\frac{\chi(p)}{p^s}.$$

Now it is readily verified that when $\sigma>1$ we have
\begin{align*}
\log L(s,\chi)&=-\sum_p\log\left(1-\frac{\chi(p)}{p^s}\right)\\
&=\sum_p\frac{\chi(p)}{p^s}+\sum_p\sum_{k=2}^{\infty}\frac{\chi(p^k)}{p^{k s}}.
\end{align*}
The second sum on the right converges locally uniformly when $\sigma>\frac12$.  Thus
$$\sum_p\frac{\chi(p)}{p^s} = \log L(s,\chi) +h(s,\chi)$$
where $h(s,\chi)$ is holomorphic for $\sigma>\frac12$.  Notice that $\log L(s,\chi)$ is analytic on the line $\sigma=1$ except when $\chi=\chi_0$ when it has a logarithmic singularity at the point $s=1$. Hence
$$\sum_{\substack{
p\\ 
\overline{p}=\overline{c_j}
}}\frac1{p^s} = \frac1{\phi(a)}\log L(s,\chi_0) + h(s,c_j)$$
where $h(s,c_j)$ is an analytic function of $s$ for $\sigma\ge 1$.  Therefore
\begin{equation}\label{e3.3}
\prod_{j=1}^{t}\sum_{\substack{
p\\ 
\overline{p}=\overline{c_j}
}}\frac1{p^s}=\frac1{\phi(a)^t}\sum_{j=0}^t(\log L(s,\chi_0))^j h_j(s)
\end{equation}   
where the $h_j(s)$ are analytic when $\sigma\ge1$ and $h_t(1)=1$.

Now on combining \eqref{e3.1}, \eqref{e3.2} and \eqref{e3.3}, we have
$$f(s)=\frac{g_1(s)}{\phi(a)^t}L(s,\chi_0)^{\frac{w}{\phi(a)}}\sum_{j=0}^t(\log L(s,\chi_0))^jh_j(s).$$
We have $L(s,\chi_0) = \zeta(s)\prod_{p|a}(1-p^{-s})$ and the Riemann zeta function $\zeta(s)$ has a simple pole at $s=1$ with residue 1 at $s=1$.  Thus
$$L(s,\chi_0)=\frac{\phi(a)g_2(s)}{a(s-1)},$$
where $g_2(s)$ is an entire function with $g_2(1)=1$.
On plugging this in to the above expression for $f(s)$, the asymptotic formula of Lemma \ref{lem3.2} follows from Lemma \ref{lem3.1}. Notice that we apply part (i) of Lemma \ref{lem3.1} when $w\ge1$ and part (ii) when $w=0$ and $t\ge1$. That the constants $C(a,\mathcal{B},t)$ and $C(a,t)$ are positive follows by observing first that, by Lemma \ref{lem3.1}, they are non-zero and then that the left hand side of the asymptotic formula is non-negative.
\end{proof}

\section{Proof of Theorem \ref{t1.1}} \label{s4}
\noindent The main analytic tool in the proof of Theorem \ref{t1.1} is Lemma \ref{lem3.2} and we will apply it to the various sets from Section \ref{s2.3}. Recall the definitions of the groups $G$ and $H$ and of the numbers $m$ and $d$ from Section \ref{s2.3}. We denote by $\mathfrak{H}$ the set of all subgroups $H$ defined in Section \ref{s2.3} for general $a$. Now as was defined in Lemma \ref{lem3.2}, we form the set 
$$\mathcal{A}(H,\mathcal{C}),$$
where
$$H=G_0\times G_1\times \cdots \times H_i \times \cdots \times G_k\in\mathfrak{H}$$ 
and $\mathcal{C}=\{\overline{c_j}\}_1^{t}$ is a sequence of length $t=2^{m-1}-1$ with elements in $G$. Moreover for a fixed primitive root  $g\pmod{p_i^{\gamma_i}}$ and a fixed odd number $e$ we have $c_j\equiv g^{e'}\pmod{p_i^{\gamma_i}}$ for some $e'$ with $e'\equiv \pm e\pmod{2^m}$.  For a fixed $H\in\mathfrak{H}$, there are only finitely many possibilities ($d^t 2^{t+m-1}$ actually) for $\mathcal{C}$. Lemma \ref{lem3.2} immediately implies that
\begin{lem}\label{lem4.1}
$$|\mathcal{A}(H,\mathcal{C})(N)|\sim C(H,\mathcal{C})\frac{N(\log\log N)^{2^{m-1}-1}}{(\log N)^{1-1/2^m}}.$$
\end{lem}
Now we need to show that the intersection of any two distinct such sets, $\mathcal{A}(H^1,\mathcal{C}^1)$ and $\mathcal{A}(H^2,\mathcal{C}^2)$, is a relatively small set. 

\begin{lem}\label{lem4.2}
We have
$$|\left(\mathcal{A}(H^1,\mathcal{C}^1)\cap\mathcal{A}(H^2,\mathcal{C}^2)
\right)(N)|\ll_a \frac{N(\log\log N)^{2^{m}-2}}{(\log N)^{1-1/4^m}}$$
\end{lem} 
\begin{proof}
If $H^1$ and $H^2$ are the same, then $\mathcal{C}^1$ and $\mathcal{C}^2$ differs in at least one element. Hence the intersection is empty. So without loss of generality, we can assume $H^1$ and $H^2$ are not the same. Then $H^1\cap H^2$ is a subgroup of $G$ with index $4^m$. Also notice the relation $$\mathcal{A}(H^1,\mathcal{C}^1)\cap\mathcal{A}(H^2,\mathcal{C}^2)\subseteq\mathcal{A}(H^1\cap H^2,\mathcal{C}^1\cup\mathcal{C}^2)$$ 
where $\mathcal{C}^1\cup\mathcal{C}^2$ is the union of the sequences $\mathcal{C}^1$ and $\mathcal{C}^2$ and hence is of length $2^m-2$. Then the desired conclusion follows from Lemma \ref{lem3.2}.
\end{proof}

Now set 
$$\mathcal{U}=\cup_{H}\cup_{\mathcal{C}}\mathcal{A}(H,\mathcal{C}),$$
where the union runs through all $H\in\mathfrak{H}$ and the corresponding sequences $\mathcal{C}$ for $H$ as defined above. We know that $\mathcal{U}\subseteq\ES_a$ from Lemma \ref{lem2.9}. 
\begin{lem}\label{lem4.3}
We have
\[
\es_a(N)-|\mathcal{U}(N)|\ll_a 
\begin{cases}
\frac{N(\log\log N)^{2^{m-1}-2}}{(\log N)^{1-1/2^m}},&\text{if $m\ge2$},\\
\frac{N(\log\log N)^{{\phi(a)}^{\phi(a)}}}{(\log N)^{1-1/2^m+1/\phi(a)}}, & \text{if $m=1$}.
\end{cases}
\]
\end{lem}

\begin{proof}
We let $\mathcal{W}(n)$ be the set of residue classes modulo $a$ in which there are at least $\phi(a)$ (counted with multiplicity) prime factors of $n$. By Lemma \ref{lem2.8} we know that
\begin{enumerate}[(i)]
\item if $|\mathcal{W}(n)|\ge d+1$, then $n\notin \ES_a$;
\item if $|\mathcal{W}(n)|=d$, then $n\notin \ES_a$ unless $\mathcal{W}(n)=H$ for some subgroup $H$ of $G$ as above.
\end{enumerate}
Let
$$\mathcal{N}(i)=\{n\in\ES_a:|\mathcal{W}(n)|=i\}$$
for $0\le i\le \phi(a)$.
From the above discussion we know $\mathcal{N}(i)$ is empty as long as $i>d$. Hence
$$\ES_a=\bigcup_{i=0}^d\mathcal{N}(i).$$
Firstly observe that by Lemma \ref{lem3.2} we have
$$\left|\left(\bigcup_{i=0}^{d-1}\mathcal{N}(i)\right)(N)\right|\ll_a \frac{N(\log\log N)^{{\phi(a)}^{\phi(a)}}}{(\log N)^{1-1/2^m+1/\phi(a)}}.$$
Now if $m=1$, then we have $\mathcal{N}(d)=\mathcal{U}$ by part (i) of Lemma \ref{lem2.9}, and if $m\ge2$, then we have by Lemma \ref{lem2.9} and Lemma \ref{lem3.2} that
$$|(\mathcal{N}(d))(N)|-|\mathcal{U}(N)|\ll_a \frac{N(\log\log N)^{2^{m-1}-2}}{(\log N)^{1-1/2^m}}.$$ 
Therefore Lemma \ref{lem4.3} follows by putting the above conclusions together.
\end{proof}
Here we bound the error term rather crudely, following from Lemma \ref{lem2.8}. Actually it can be refined substantially by a generalisation of Lemma \ref{lem2.4}, which is, however, not pertinent to the purpose of the current paper. 

Now Theorem \ref{t1.1} follows from Lemma \ref{lem2.1}, Lemma \ref{lem4.1}, Lemma \ref{lem4.2} and Lemma \ref{lem4.3}. It should be noted that the leading constant $C(a)$ appearing in Theorem \ref{t1.1}, can be traced back explicitly in our arguments, but is inevitably messy, would require some non-trivial expenditure of effort and would not give any further insights into our problem.

\end{document}